\documentclass[a4paper,leqno,twoside]{amsart}
\usepackage{amsmath,amssymb,amsthm,enumerate,hyperref,color,esint}
\numberwithin{equation}{section}
\newtheorem{theorem}{Theorem}[section]

\newtheorem{proposition}[theorem]{Proposition}
\newtheorem{lemma}[theorem]{Lemma}
\newtheorem{corollary}[theorem]{Corollary}
\newtheorem{remark}[theorem]{Remark}

\newcommand{\rad}{{\text{\upshape rad}}}
\newcommand{\loc}{{\text{\upshape loc}}}

\def\L{{\Lambda}}
\def\l{{\lambda}}
\def\a{{\alpha}}

\newcommand{\R}{\mathbb{R}}

\definecolor{darkgreen}{rgb}{0.0, 0.5, 0.2} 
\definecolor{purple}{rgb}{0.5, 0.0, 0.5}
\newcommand{\AL}{\color{purple}}
\newcommand{\F}{\color{magenta}}
\newcommand{\taglia}{\color{cyan}}

\newcommand{\n}{\mathop{N}}

\newcommand{\remove}[1]{}

\def\sideremark#1{\ifvmode\leavevmode\fi\vadjust{\vbox to0pt{\vss
 \hbox to 0pt{\hskip\hsize\hskip1em
 \vbox{\hsize2.1cm\tiny\raggedright\pretolerance10000
  \noindent #1\hfill}\hss}\vbox to15pt{\vfil}\vss}}}%
\newcommand{\edz}[1]{\sideremark{#1}}

\newtheorem*{theorem*}{Theorem}

\begin{document}

\title[Singular eigenvalue problem] {On a singular eigenvalue problem and its applications in computing the Morse index of solutions to semilinear PDE's, Part II}
\author[A.~L.~Amadori, F.~Gladiali]{Anna Lisa Amadori$^\dag$,  Francesca Gladiali$^\ddag$}
\thanks{This work was supported by Gruppo Nazionale per l'Analisi Matematica, la Probabilit\`a e le loro Applicazioni (GNAMPA) of the Istituto Nazionale di Alta Matematica (INdAM). {The second author is supported by Prin-2015KB9WPT}}
\date{\today}
\address{$\dag$ Dipartimento di Scienze e Tecnologie, Universit\`a di Napoli ``Parthenope", Centro Direzionale di Napoli, Isola C4, 80143 Napoli, Italy. \texttt{annalisa.amadori@uniparthenope.it}}
\address{$\ddag$ Dipartimento di Chimica e Farmacia, Universit\`a di Sassari, via Piandanna 4, 07100 Sassari, Italy. \texttt{fgladiali@uniss.it}}

\begin{abstract}
By  using a characterization of the Morse index and the degeneracy in terms of a singular one dimensional eigenvalue problem given in \cite{AG-part1}, we give a lower bound for the Morse index of radial solutions to H\'enon type problems 
	\[
	\left\{\begin{array}{ll}
	-\Delta u = |x|^{\a}f(u) \qquad & \text{ in } \Omega, \\
	u= 0 & \text{ on } \partial \Omega,
	\end{array} \right.
	\]
	where $\Omega$ is a bounded radially symmetric domain of $\R^N$ ($N\ge 2$), $\a>0$ 
	and $f$ is a real function.   From this estimate we get that the Morse index of nodal radial solutions to this problem goes to $\infty$ as $\a\to \infty$.
	Concerning the real H\'enon problem, $f(u)= |u|^{p-1}u$, we prove radial nondegeneracy, we show that the radial Morse index is equal to the number of nodal zones and we get that a least energy nodal solution is not radial.
\end{abstract}

\maketitle



\section{Introduction}
In this paper we estimate the Morse index of radial solutions to 
\begin{equation} \label{general-f-H}
\left\{\begin{array}{ll}
-\Delta u = |x|^{\alpha}f(u) \qquad & \text{ in } \Omega, \\
u= 0 & \text{ on } \partial \Omega,
\end{array} \right.
\end{equation}
where $\Omega$ is a bounded radially symmetric domain of $\R^N$, with $N\ge 2$, $\a\geq 0$ is a real parameter
and 
$f $ is a real function. We will consider weak and classical solutions. When $\a=0$ problem \eqref{general-f-H}
becomes autonomous 
\begin{equation} \label{general-f-auto}
\left\{\begin{array}{ll}
-\Delta u = f(u) \qquad & \text{ in } \Omega, \\
u= 0 & \text{ on } \partial \Omega,
\end{array} \right.
\end{equation} 
and we recover, from a different point of view, an already known estimate on the Morse index of radial solutions to \eqref{general-f-auto}, see \cite{AP}, \cite{BDG} and \cite{DIPa}.

Since this paper is based on the Morse index of a solution we  recall its definition and its relevance in the study of P.D.Es.  Taken a weak solution $u\in H^1_0(\Omega)$ to \eqref{general-f-H} we introduce the associated linearized operator  
	\begin{align}
	\label{linearized-intro}
	L_u(\psi)&:=-\Delta \psi-|x|^\a f'(u)\psi
\intertext{ and the associated quadratic form }
\label{forma-quadratica}
{\mathcal Q}_u(\psi)& :=\int_\Omega \left(|\nabla \psi|^2 -|x|^\a f'(u)\,\psi^2\right) dx
\end{align}
In order to give sense to $L_u$ and $\mathcal Q_u$ 
 we will consider weak solutions $u\in H^1_0(\Omega)$ to \eqref{general-f-H} under the hypotheses
\begin{enumerate}[\bf{H.}1] 
	\item 
	$f \in W^{1,1}_{\loc}(\R)$,
	\item $f'(u)\in L^{\infty}(\Omega)$.
\end{enumerate}
	Assumptions {{H.}1} and {{H.}2} are needed to give a sense to $f'(s)$ and to the weak formulation to  \eqref{general-f-H} and  \eqref{linearized-intro}
	 and $Q_u(\psi)$ and to recover compactness of the linear operator $L_u$, so to use the eigenvalue theory for compact operators. 
It is easily seen that if $f\in C^1(\R)$ and $u$ is a classical solution  then both assumptions hold.
Besides assumption H.2 is satisfied by every radial weak solution if $f$ satisfies some stricter condition, like for instance 
\begin{enumerate}[\bf{H.1'}] 
	\item $f\in W^{1,\infty}_{\loc}(\R)$ and $|f(s)|\le C\left( 1 + |s|^p\right)$ when $s$ is large, for some constant $C$ and $p\in\left(1, \frac{2N-2+\a}{2+\a}\right)$, or $p>1$ if $N=2$.
\end{enumerate}
See Remark \ref{remark:H'}.
	The hypothesis H.1' has been introduced by Ni \cite{Ni}, together with some other ones, to prove existence of  radial solutions to \eqref{general-f-H} and in particular  to the real H\'enon problem. \\
In some results we will also assume that $f$ satisfies
\begin{enumerate}[\bf {H.}3]
	\item $f'(s)>f(s)/s $, $s\neq 0$.
\end{enumerate}

\

 Given a weak solution $u$ the Morse index of $u$, that we denote by $m(u)$, is the maximal dimension of a subspace of $H^1_0(\Omega)$ in which the quadratic form $Q_u$ is negative defined, or equivalently, since $L_u$ is a linear compact operator, is the number of the negative eigenvalues of $L_u$ in $H^1_0(\Omega)$, counted with multiplicity and when $u$ is a radial solution the radial Morse index of $u$, called $m_\rad(u)$ is the number of the negative eigenvalues of $L_u$ in $H^1_{0,\rad}(\Omega)$ (the subspace of $H^1_0(\Omega)$ given by radial functions). 

The knowledge of the Morse index of a solution $u$ has important applications. Let us recall that a change in the Morse index, gives existence of other solutions that can be obtained by bifurcation and can give rise to the so called symmetry breaking phenomenon, that in the contest of the H\'enon problem has been highlighted by 
\cite{SSW} for  a least energy solution.
In the variational setting, indeed, there is a direct link between the second derivative of the 
	functional associated to \eqref{general-f-H} and the quadratic form $Q_u$ related to its linearization, and a change in the Morse index immediately produces a change in the critical groups, giving existence of bifurcating solutions; we refer to \cite{BSW} for the definition of critical groups, and their relation with the Morse index.
But also when the problem does not have a variational structure, as for instance when $f$ is supercritical, a change in the Morse index implies a bifurcation result, via the Leray Schauder degree, see \cite{AM}.  An application of this type can be found in  \cite{AG-henon}, dealing with positive solutions of the H\'enon problem in the ball. 
\\
 The knowledge of the Morse index also allows to produce nonradial solutions by minimization, as done in \cite{GI}, dealing with the Lane-Emden problem in the disk {and in \cite{A}, \cite{AG18-2} in the case of the H\'enon problem. 

\

The study of the Morse index of nodal radial solutions has been tackled for the first time by Aftalion and Pacella, in \cite{AP}, dealing with autonomous problem of the type \eqref{general-f-auto} with $f\in C^1$. They proved that the linearized operator $L_u$ has at least $N$ negative eigenvalues whose corresponding eigenfunctions are non radial and odd with respect to $x_i$. Adding the first eigenvalue, which is associated to a radial, positive eigenfunction, one gets $m(u)\geq N+1$. 
Next denoting by $m$ the number of the nodal zones,  namely the connected components of  $\{x\in \Omega\ : \ u(x)\neq 0\}$, the paper \cite{BDG}
proved a similar estimate, precisely that $m(u)\geq (m-1)(N+1)$. In this case $f$ is  
absolutely continuous, but a restriction on its growth is imposed so that \eqref{general-f-auto} has a variational structure. 
 Next \cite{DIPa} established the following lower bound
	\begin{theorem*}[2.1 in \cite{DIPa}]
			Let $f\in C^1(\R)$, and $u$ 
		be a classical radial solution to \eqref{general-f-auto} with $m$ nodal zones. 
		Then 
		\begin{align*}
		& m_{\rad}(u)  \ge m -1 , \;
		& m(u)  \ge  
		(m-1)(1+N).
		\intertext{   If in addition $f$ fulfills H.3 , then  }
		& m_{\rad}(u) \ge  m , \;
		& m(u)  \ge  
		m+(m-1)N.
		\end{align*}
	\end{theorem*}

All the mentioned estimates are achieved using the directional derivatives of the solution $u$, namely $\frac{\partial u}{\partial x_i}$, to obtain information on the eigenfunctions and eigenvalues of $L_u$, since $L_u\big(\frac{\partial u}{\partial x_i}\big)=0$ and cannot be adapted to deal with nonautonomous nonlinearities. 

Concerning the Morse index of nodal least energy solutions we quote \cite{BCW} and \cite{BW}, dealing with variational problems. 
Coming to nonautonomous problems of H\'enon type \eqref{general-f-H} we quote a recent paper by Dos Santos and Pacella \cite{dosSP} which proved
that any nodal radial solution in a radially symmetric planar domain satisfies $m(u)\ge 3$ for any $\alpha>0$ and $m(u)\ge 3+ \alpha$ when $\alpha$ is an even integer. 
Under the additional assumption H.3,  also the paper \cite{dosSP} furnishes an improved estimate claiming that $m(u)\ge m+2$ for any $\alpha>0$ and $m(u)\ge m+2+\alpha$ when $\alpha$ is an even integer. The proof relies on a suitable transformation which relates solutions to \eqref{general-f-H} to solutions of an autonomous problem of type \eqref{general-f-auto}, to which \cite[Theorem 2.1]{DIPa} can be applied.

Here we improve the results in \cite{dosSP} in two different directions: from one side we provide a higher lower bound in the planar case, from the other  we include the case of higher dimensions.
Letting $\left[\frac{\alpha}{2}\right] = \max \left\{ n\in {\mathbb Z} \, : \, n\le \frac{\alpha}{2}\right\}$ stand for the integer part of $\frac{\alpha}{2}$, and  $N_j=\frac{(N+2j-2)(N+j-3)!}{(N-2)!j!}$  for the multiplicity of the $j^{th}$ eigenvalue of the Laplace-Beltrami operator,
our estimates state as follows:

\begin{theorem}\label{morse-estimate-H}
	Assume that $\a\ge 0$ and $f$ satisfies H.1, and take $u$ a radial weak solution to \eqref{general-f-H}  with $m$ nodal zones satisfying {H.2}. 
	Then 
	\begin{align}	\label{general-radial-morse-H} 
	m_{\rad}(u)  &\ge m -1 , \\
	\label{general-morse-estimate-H}
	m(u)  & \ge   m_{\rad}(u) + (m-1)\sum\limits_{j=1}^{[\frac{2+\alpha}{2}]}N_{j} \ge (m-1)\sum\limits_{j=0}^{[\frac{2+\alpha}{2}]}N_{j} 
	\\
	\nonumber & = \left\{ \begin{array}{ll}
	(m-1)(1+N) & \mbox{if $0\le \a<2$, or} \\
	 (m-1) \left(1+N + \sum\limits_{j=1}^{[\frac{\alpha}{2}]}  N_{j+1} \right)
	& \text{ if } \a\geq 2 .
	\end{array}\right. 
\end{align}        
	If in addition $f$ fulfills 
	{H.3}, then 
	\begin{align}\label{radial-morse-f(u)-H} 
	 m_{\rad}(u) &\ge  m , \\
	\label{morse-estimate-f(u)-H}
	 m(u)  &\ge  m_{\rad}+	(m-1)\sum\limits_{j=1}^{[\frac{2+\alpha}{2}]}N_{j} \ge m+(m-1)\sum\limits_{j=1}^{[\frac{2+\alpha}{2}]}N_{j} 
	\\
	\nonumber &  = \left\{ \begin{array}{ll}
	m+(m-1)N & \mbox{if $0<\a<2$, or} \\
	m+(m-1) \left(N + \sum\limits_{j=1}^{[\frac{\alpha}{2}]}  N_{j+1}\right)
	& \text{ if } \a\geq 2 .
	\end{array}\right. \end{align}      
\end{theorem}
The proof of Theorem \ref{morse-estimate-H} relies on a transformation of the radial variable which, like the one in \cite{dosSP}, brings radial solutions to problem \eqref{general-f-H} into solutions of a suitable autonomous o.d.e (see \cite[Sect. 4.1]{AG-part1}). The main difference in our approach is  that we compute the Morse index starting from  a  {\it singular} eigenvalue problem studied in the  the first part of this paper, \cite{AG-part1}. In that way the core of the proof stands in an estimate of the
{\it singular}  eigenvalues given in Proposition \ref{stima-nu-H}.
Such estimate, together with \cite[Corollary 4.11]{AG-part1},  
allows  to obtain informations also on the Morse index in symmetric spaces and has interesting implications on the multiplicity of solutions, as discussed with more details at the end of Section  \ref{sec:4}.

	Let us remark by now an immediate but interesting consequence of estimate \eqref{general-morse-estimate-H}.
	\begin{corollary}\label{morse-index-alpha-infinity}
	Assume that $\a\ge 0$ and $f$ satisfies H.1, and take $u$ a radial weak solution to \eqref{general-f-H}  with $m\ge 2$ nodal zones satisfying {H.2}.
	Then the Morse index of $u$ goes to infinity  as $\a \to +\infty$.
	\end{corollary}
		This result holds only for sign-changing solutions and indeed cannot be true in the case of positive ones, as shown in \cite{AG-henon} where the positive solution has Morse index one for every value of $\a>0$, for some particular choice of the function $f$.
		\\
               After this paper was finished we came to know that Corollary \ref{morse-index-alpha-infinity} was previously presented in the paper \cite{LWZ} for $p$-homogeneous nonlinearities. Their result generalizes also to the case of systems. Following an idea of \cite{BW} they transform problem \eqref{general-f-H} into an equivalent one and they perform a blow-up analysis as $\a\to \infty$. A Liouville theorem for the limiting problem, included in the paper, then implies the result. 
                  Let us observe that the strategy of \cite{LWZ} is complementary to ours. Indeed our result does not relies on an asymptotic analysis and produces informations for every fixed value of $\a$. 

\

We conclude our paper by dealing with the particular case of power-type non-linearity, i.e.~ with the H\'enon problem 
\begin{equation}\label{H-intro}
\left\{\begin{array}{ll}
-\Delta u = |x|^{\alpha}|u|^{p-1} u \qquad & \text{ in } \Omega, \\
u= 0 & \text{ on } \partial \Omega ,
\end{array} \right.
\end{equation}
 that has been introduced by H\'enon in \cite{H} to study stellar clusters. Attention to this problem has been brought by the existence result in \cite{Ni} and by the breaking of symmetry of the ground state solution in \cite{SSW}. After that the H\'enon problem attracted the attention of many authors, and the interested reader
 can see among others the following ones \cite{AG-henon,AG18,AG18-2,BW,BW2, CP, CLP18,CD17,GG,KW,PS,WY}.
We recall that a solution $u$ is said radially degenerate if the linearized equation $L_u(\psi)=0$ admits a radial solution in $H^1_0(\Omega)$. By investigating the {\it singular} radial eigenvalues related to \eqref{H-intro}, we are able to show that
\begin{theorem}\label{prop:ultimo-H}
	Let $\alpha\ge 0$ and $u  \in H^1_0(\Omega)$ a radial solution to \eqref{H-intro} with $m$ nodal zones. Then $u$ has radial Morse index $m$ and is radially non-degenerate. 
\end{theorem}

Theorem \ref{prop:ultimo-H} includes also the Lane-Emden problem ($\alpha=0$). For that problem both the radial non-degeneracy and the value of the radial Morse index had already been obtained in \cite{HRS} with a completely different approach. Their proof adapts to deal with some non-autonomous problems, but their assumptions do not include the H\'enon problem and they only handle variational problems (i.e.~subcrictical exponents). 

Beside for the H\'enon problem an easy corollary follows from the Morse index estimate in Theorem \ref{morse-estimate-H} 
\begin{corollary}\label{cor-ultimo}
Let $\alpha\ge 0$ and $1<p<\frac{N+2}{N-2}$ if $N\ge 3$, or $1<p$ in dimension $N=2$.
A least energy nodal solution to \eqref{H-intro} is not radial.
\end{corollary}
This result follows easily by Morse index considerations and was previously known only for small values of $\a$ in \cite{BDRT}. 
It generalizes previous results for autonomous problem in \cite{AP} and \cite{BDG} and can be proved for more general nonlinearities when problem \eqref{general-f-H} admits a variational structure (see as an example assumptions $f_1, f_2,f_3, f_4$ in \cite{BW}), by relying on Theorem \ref{morse-estimate-H}.
On the other hand the same symmetry breaking phenomenon was already proved for the ground state solution to \eqref{H-intro}  in \cite{SSW},  by estimating the energy of the positive radial solution, but it holds only for large values of $\a$.

\remove{\taglia This result follows by Morse index considerations and 
can be obtained from \eqref{general-morse-estimate-H} for more general nonlinearities when problem \eqref{general-f-H} admits a variational structure (see as an example assumptions $f_1, f_2,f_3, f_4$ in \cite{BW} or in \cite{AG-part1}) and generalizes previous results for autonomous 
problem in \cite{AP} and \cite{BDG}. Let us recall that the same symmetry breaking phenomenon for positive least energy solutions to \eqref{H-intro} was proved in \cite{SSW}, in this subcritical setting, by estimating the energy of the positive radial solution, but only for large values of $\a$.}

Finally we mention that, starting from the Morse index formula in \cite[Proposition 1.4]{AG-part1}, Theorem \ref{prop:ultimo-H} and the estimates of the {\it singular} eigenvalues obtained in Proposition \ref{stima-nu-H}, we are able to compute the Morse index of radial solutions to \eqref{H} when the parameter $p$ goes to the end of the existence range, by means of   a careful investigation of the asymptotic behaviour of the solution as well of the {\it singular} radial eigenvalues and eigenfunctions that we defer to the papers \cite{AG18} and \cite{AG18-2}.

 \section{Preliminaries}\label{sec:2}

 In this section we give all the notations we need in the following, we introduce the 
 singular eigenvalue problems that have been the subject of \cite{AG-part1} and we recall their relation with the Morse index of a solution $u$ to \eqref{general-f-H} that we need to prove the main results. 
Since this paper is the sequel of \cite{AG-part1} we suggest to read the first part where some properties of the singular eigenvalues and eigenfunctions are proved.	\\
In the following $\Omega$ denotes a bounded radially symmetric domain of $\R^{N}$,  while $B=\{x\in\R^N \, : \, |x|<1\}$ is the unit ball. In the end of this section we will focus on the case when $\Omega=B$  since the case of the annulus is easier and can be deduced from this one. \\
	We shall make use of the following functional spaces: $
   C^1_0(\Omega):= \{ v:\Omega\to \R \, : \, v $ differentiable, $\nabla v $ continuous and the support of $ v $ is a compact subset of $\Omega \} $; for any $p>1$ we let 
$  L^p(\Omega)$ be the usual Lebesgue spaces;
while $H^1(\Omega)$ and $H^1_0(\Omega)$ are the Sobolev spaces, namely 
$  H^1(\Omega) := \{ v\in L^2(\Omega) \, : \, v $  has first order weak derivatives $ \partial_{i}v \text{ in } L^2(\Omega) \text{ for }i=1,\dots,N \}$; $  H^1_0(\Omega) := \{ v\in H^1(\Omega) \, : \, v(x) =0 \mbox { if } x\in \partial \Omega \} $; and $H^1_\rad(\Omega)$ and $  H^1_{0,\rad}(\Omega)$ are the subspaces given by radial functions, namely 
$   H^1_\rad(\Omega):= \{v\in H^1(\Omega)\, : \, v \text{ is radial }\}$;
$  H^1_{0,\rad}(\Omega) := H^1_0(\Omega)\cap  H^1_\rad(\Omega)$.

Following  \cite{AG-part1} we use some singular eigenvalues associated to the linearized operator $L_u$ to characterize the Morse index  of a solution $u$ to \eqref{general-f-H}. 
To define them we need some weighted Lebesgue and Sobolev spaces that we denote by
   \begin{align*}
  {\mathcal L} & := \{\psi:  \Omega\to \R\, : \,  \psi \text{ measurable and s.t } \int_\Omega |x|^{-2}\psi^2\, dx<\infty\},
\\ 
   \mathcal{H} &:=H^1(\Omega)\cap  {\mathcal L}, \quad \mathcal{H}_0 :=H_0^1(\Omega)\cap \mathcal L,
\quad
\mathcal{H}_{0,\rad} :=  \mathcal{H}\cap H^1_{0,\rad}(\Omega) , 
\end{align*}
$\mathcal L$ is a Hilbert space with the scalar product $\int_\Omega |x|^{-2}\eta\varphi\ dx$, so that 
\begin{equation}\label{scalar-H}
\eta\underline{\perp}\varphi\  \ \Longleftrightarrow \int_\Omega |x|^{-2}\eta\varphi\ dx=0  \  \ \text{ for }\eta, \varphi\in  {\mathcal L}.
\end{equation}

Next we introduce the singular eigenvalues that have been studied in \cite[Section 3]{AG-part1} and we let 
\begin{equation}\label{eq:primo-autov}
\widehat \L_1:=
 \inf \left\{ \frac{Q_u(\psi)}{\int_{\Omega} |x|^{-2}\psi^2(x)\, dx}  : \, \psi\in  {\mathcal H}_0\setminus\{0\}, \right\}
\end{equation}
where $Q_u(\psi)$ is as defined in \eqref{forma-quadratica}. This first singular eigenvalue $\widehat \L_1$ is attained, when $\widehat \L_1<\left(\frac{N-2}2\right)^2$ 
 at a function $\varphi_1\in \mathcal H_0$. Iterating, when $\widehat \L_{i-1}<\left(\frac{N-2}2\right)^2$ and it is attained
at a function $\varphi_{i-1}\in \mathcal H_0$, we can then define the subsequent eigenvalue
\begin{equation}\label{i+1-singular}
\widehat \L_{i}:= \inf \left\{ \frac{Q_u(\psi)}{\int_{\Omega} |x|^{-2}\psi^2(x)\, dx}  : \, \psi\in  {\mathcal H}_0\setminus\{0\}, \,  w\underline{\perp} \varphi_1,\dots, \varphi_{i-1} \right\},
\end{equation}
where the orthogonality stands for the orthogonality in $\mathcal L$. Again $\widehat \L_i$ is attained as far as it 
satisfies $\widehat \L_i<\left(\frac{N-2}2\right)^2$. Every eigenfunction $\varphi_i\in \mathcal H_0$ associated with $\widehat \L_i$ is a weak solution to the singular eigenvalue problem
\begin{equation}\label{eq:autovalori-singolari}
\left\{\begin{array}{ll}
-\Delta \varphi_i - |x|^\a f'(u) \varphi_i = \frac{\widehat \L_i}{|x|^2}\varphi_i \qquad & \text{ in } \Omega, \\
\varphi_i= 0 & \text{ on } \partial \Omega,
\end{array} \right.
\end{equation}
meaning that it satisfies
\[
\int_{\Omega}\nabla \varphi_i\nabla \phi -|x|^\a f'(u) \varphi_i \phi dx=\widehat \L_i \int_\Omega |x|^{-2}\varphi_i\phi dx
\]
 for every $\phi\in \mathcal H_0$. We need also the radial version of the singular eigenvalues and so we let
\begin{equation}\label{radial-singular-uno}
\widehat \L_{1}^{\rad}:=\inf \left\{ \frac{Q_u(\psi)}{\int_{\Omega} |x|^{-2}\psi^2(x)\, dx}  : \, \psi\in  {\mathcal H}_{0,\rad}\setminus\{0\}\right\}
\end{equation}
which is attained when $\widehat \L_{1}^{\rad}<\left(\frac{N-2}2\right)^2$ at a function $\varphi_1^\rad\in \mathcal H_{0,\rad}$ and, as before, whenever $\widehat \L_{i-1}^\rad<\left(\frac{N-2}2\right)^2$ and it is attained
at a function $\varphi_{i-1}^\rad\in \mathcal H_{0,\rad}$, we can then define the subsequent eigenvalue
\begin{equation}\label{radial-singular-i}
\widehat \L_{i}^{\rad}:=\inf \left\{ \frac{Q_u(\psi)}{\int_{\Omega} |x|^{-2}\psi^2(x)\, dx}  : \, \psi\in  {\mathcal H}_{0,\rad}\setminus\{0\}, \,  \psi\underline{\perp}\varphi_1^{\rad},\dots,\varphi_{i-1}^{\rad} \right\}.
\end{equation}
The interest in the singular eigenvalues stands in the fact that, even for semilinear problems more general than \eqref{general-f-H},  the Morse index of any solution $u$ can be computed by counting, with multiplicity, the singular eigenvalues $\widehat \L$, while the radial Morse index of a radial solution $u$ is the number of negative singular radial eigenvalue $\widehat\L^{\rad}$, see \cite[Proposition 1.1]{AG-part1}. 
	Further when $u$ is radial they 
have the good property 
a decomposition along radial and angular part holds. We collect here into one statement (adapted to the particular case \eqref{general-f-H}) the main results in \cite{AG-part1} about this topic
recalling that $\l_j$ are the eigenvalues of the Laplace Beltrami operator on the sphere $S^{N\!-\!1}$, namely 
$-\Delta_{S^{N\!-\!1}} Y_j=\l_jY_j$ for 
\[\l_j=j(N-2+j)\]
and whose multiplicity is 
\[N_j:=\frac{(N+2j-2)(N+j-3)!}{(N-2)!j!}\]
and  $Y_j=Y_j(\theta)$ are the eigenfunctions of $-\Delta_{S^{N-1}}$ associated with $\l_j$ and they are known as Spherical Harmonics.

\begin{proposition}\label{general-morse-formula} 
Assume that $\a\geq 0$ and $f$ satisfies $H.1$ and take $u$ a radial weak solution to \eqref{general-f-H} satisfying $H.2$.	
	Then its radial Morse index $m_{\rad}$ is the number of negative eigenvalues $\widehat\L^{\rad}_i$ according to \eqref{radial-singular-i}, and its 
		Morse index is given by 
	\begin{equation}\label{tag-2}\begin{split}
	m(u)= & \sum\limits_{i=1}^{m_{\rad}}
	 \sum\limits_{j=0}^{\lceil J_i -1\rceil } N_j  \qquad \qquad  \mbox{where} \\
	J_i= & \sqrt{\left(\frac{\n-2}{2}\right)^2-\widehat\L^{\rad}_i}-\frac{\n-2}{2}  
	\end{split}\end{equation}
	 and $\lceil t \rceil = \min\{ k\in {\mathbb Z} \, : \, k\ge t\}$ stands for the ceiling function.
	 \\
		Besides the negative singular eigenvalues are $\widehat\L = \widehat\L^{\rad}_i +\lambda_j$ 
		and the related eigenfunctions are, in spherical coordinates
		\begin{equation}\label{decomp-autofunz-f(u)} \psi(x)=\widehat \psi_i^{\rad}(r)Y_j(\theta), \end{equation} 
		where $\widehat \psi_i^{\rad}$ is an eigenfunction related to $\widehat\L^{\rad}_i$.	
	\end{proposition}

In the radial setting problem \eqref{general-f-H} is related to an autonomous one by means of the transformation
\begin{equation}\label{transformation-henon}
 t=r^{\frac{2+\a}{2}} ,\qquad w(t)=u(r) ,
\end{equation}
 which has been introduced in \cite{GGN} and maps any  radial solution $u$ of 
 \eqref{general-f-H}  into a  solution $w$ of  
\begin{equation}\label{ode-sect-2}
- \left(t^{M-1} w^{\prime}\right)^{\prime}= \left(\frac{2}{2+\a}\right)^2 t^{M-1} f(w)  , \qquad   0<t< 1,
\end{equation}
where
\begin{align}\label{Malpha}
M & = M(N,\alpha):= \frac{2(N+\alpha)}{2+\alpha}\in[2,N] 
\end{align} 
with some boundary conditions that depends on the case when $\Omega$ is a ball and when $\Omega$ is an annulus. As explained in \cite{AG-part1} the Morse index of $u$ can be computed in terms of some singular eigenvalues associated with the linearization to \eqref{ode-sect-2} at $w$, if $u$ and $w$ are related by \eqref{transformation-henon}. 
Since the topic is slightly different when $\Omega$ is a ball or an annulus, we focus here on the case when $\Omega$ is the unit ball since 
the case of the annulus can be easily deduced from this one. \\
In this case the function $w$ satisfies the boundary conditions
\begin{equation}\label{bc}
w'(0)=0, \qquad w(1)=0\end{equation}
and to deal with the singular eigenvalues
for any $M\ge 2$, we define
    \begin{align*}
  L^2_M&: = \{v:(0,1)\to\R\, : \, v \text{ measurable and s.t. } \int_0^1 t^{M-1} v^2 dt < +\infty\} ,
  \\
  H^1_M& : = \{v\in L^2_M \, : \, \text{ $v$ has a first order weak derivative $v'$ in }L^2_M \},
  \\
 H^1_{0,M} &: = \left\{ v\in H^1_M \, : \, v(1)=0\right\}  .
\end{align*}
The Lebesgue space $  L^2_M$ is a Hilbert space endowed with the scalar product $\langle v,w\rangle_M = \int_0^1 t^{M-1} v \, w \, dt ,$
	which yields the orthogonality condition
	\[ v \perp_M w \, \Longleftrightarrow \, \int_0^1 t^{M-1} v \, w \, dt = 0 .\]
The spaces $H^1_M$ and $H^1_{0,M}$ can be seen as generalizations of the spaces of radial functions  $H^1_{\rad}(B)$ and $H^1_{0,\rad}(B)$ because when $M=N$ is an integer then $H^1_N$ is actually equal to $H^1_{\rad}(B)$ by \cite[Theorem 2.2]{DFetal}. 
Next we say that $w\in  H^1_{0,M}$ is a weak solution to \eqref{ode-sect-2} and \eqref{bc} if
	\begin{equation}\label{lane-emden-radial-weak-sol}
	\int_0^1t^{M-1}w' \varphi' dt=\left(\frac{2}{2+\a}\right)^2\int_0^1 t^{M-1} f(w)  \varphi \ dt
	\end{equation}
	for every $\varphi \in H^1_{0,M}$. 

In the spaces $H^1_{0,M}$ we  generalize the {\em classical} radial eigenvalues of $L_u$
considering the Sturm-Liouville eigenvalue problem associated with the linearization of \eqref{ode-sect-2}, namely, if $w$ is a solution to \eqref{ode-sect-2} we consider
\begin{equation}\label{radial-eigenvalue-problem-M}
\left\{\begin{array}{ll}
-\left(t^{M-1}\psi_i'\right)' -t^{M-1} \left(\frac{2}{2+\a}\right)^2 f'(w)\psi_i =t^{M-1} \nu_i\psi_i & \text{ for } t\in ( 0,1)\\
\psi_i'(0)=0 , \quad \psi_i(1)=0  .
\end{array} \right.
\end{equation}
By weak solution to \eqref{radial-eigenvalue-problem-M} we mean a $\psi_i\in H^1_{0,M}$  such that 
\begin{equation}\label{radial-eigenvalue-weak-sol}
\int_0^1 t^{M-1}\left(\psi_i'\varphi' -  \left(\frac{2}{2+\a}\right)^2 f'(w) \psi_i \varphi\right) dt = \nu_i\int_0^1 t^{M-1}\psi_i \varphi \ dt .
\end{equation}
 for every $\varphi \in H^1_{0,M}$. 
Under assumptions H.1 and H.2 letting 
\begin{equation}\label{forma-quadratica-a-rad}
\mathcal Q_{w} : H^1_{0,M}\to \R, \qquad \mathcal Q_{w}(\psi)=\int_0^1 t^{M-1}\left(|\psi'|^2-\left(\frac{2}{2+\a}\right)^2 f'(w)\psi ^2\right) dt
\end{equation}
these eigenvalues $\nu_i$ can be defined using their min-max characterization, 
\[
\nu_1 :=\min_{\substack{\psi \in  H^1_{0,M}\\ w\neq 0}}
 \frac{ \mathcal{Q}_{w}(\psi)}{\int_0^1 t^{M-1} \psi^2(t)\, dt}  ,
\]
and for $i\ge 2$
	\begin{equation}\label{Rayleigh-rad-M}
\nu_i :=\min_{{\substack{\psi\in  H^1_{0,M}\\ \psi\neq 0\\ \psi\perp_{M} \{ \psi_1,\dots,\psi_{i-1}\} }}} \frac{ \mathcal Q_{w}(\psi)}{\int_0^1 t^{M-1}\psi^2(t)\, dt} 
=\min_{\substack{W\subset H^1_{0,M} \\{\mathrm{dim}} W=i}} \max_{\substack{\psi\in W \\ \psi\neq 0}}\frac{ \mathcal Q_{w}(\psi)}{\int_0^1 t^{M-1}\psi^2(t)\, dt}  .
\end{equation}
where $\psi_j$ is an eigenfunction corresponding to $\nu_j$ for $j=1,\dots, i-1$.

\

Finally, for any $M\ge 2$ 
we define the weighted Lebesgue and Sobolev spaces
\begin{align*}
{\mathcal L}_M &  :=\{ v: (0,1)\to \R \, : \, v  \text{ measurable and s.t } \int_0^1 t^{M-3}w^2\, dt<\infty\} ,
\\
\mathcal{H}_M & :=H^1_M\cap {\mathcal L}_M , \quad 
 \mathcal{H}_{0,M} :=H^1_{0,M}\cap {\mathcal L}_M .
 \end{align*}
$\mathcal L_M$ is an Hilbert space 
with the scalar product $\int_0^1 t^{M-3}\eta\varphi\ dt$, so that 
\begin{equation}\label{scalar-HM}
\eta\underline \perp_{M}\varphi\  \ \Longleftrightarrow   \int_0^1 t^{M-3}\eta\varphi\ dt=0  \  \ \text{ for }\eta, \varphi\in {\mathcal L}_M.
\end{equation}
Using these spaces we generalize the radial singular eigenvalues $\widehat \L_i^\rad$ looking at the singular Sturm-Liouville problem 
	\begin{equation}\label{radial-singular-problem-M}
	\left\{\begin{array}{ll}
	- \left(t^{M-1} \psi'\right)'- t^{M-1} \left(\frac{2}{2+\a}\right)^2 f'(w)\, \psi = t^{M-3} \widehat{\nu}_i  \psi & \text{ for } t\in(0,1)\\
	\psi\in  \mathcal H_{0,M}
	\end{array} \right.
\end{equation} 
	with $\widehat{\nu}_i\in \R$.  
	A weak solution to \eqref{radial-singular-problem-M} is $\psi\in \mathcal{H}_{0,M}$ such that 
	\begin{equation}\label{weak-radial-general} 
	\int_0^1 t^{M-1}\left( \psi_i'\varphi' -   \left(\frac{2}{2+\a}\right)^2 f'(w)\,\psi_i \varphi\right) dt=\widehat{\nu}_i\int_0^1 t^{M-3} \psi_i\varphi\, dt \end{equation}
	for any $\varphi\in \mathcal{H}_{0,M}$.
We say that $\widehat\nu_i$ is a singular eigenvalue if there exists $\psi_i\in \mathcal{H}_{0,M}\setminus\{0\}$ that satisfies \eqref{weak-radial-general}. Such $\psi_i$ will be called singular eigenfunction.
If $M=N$ is an integer then $\mathcal{H}_{0,M}=\mathcal{H}_{0,\rad}$ and $\widehat{\nu}_i=\widehat \L^{\rad}_i$ are the radial singular eigenvalues according to the previous definition. The eigenvalues $\widehat \nu_i$ 
can be defined letting 
\[
\widehat{\nu}_1:=\inf_{\substack{\psi\in\mathcal{H}_{0,M}\ \psi\neq 0}} \frac{ \mathcal Q_{w}(\psi)}
{\int_0^1 t^{M-3}\psi ^2\, dt},  
\]
This first eigenvalue $\widehat \nu_1$ is attained when $\widehat \nu_1<\left(\frac{M-2}2\right)^2$ at a function $\psi_1\in \mathcal H_{0,M}$ which is a weak solution to \eqref{radial-singular-problem-M}. Iterating, when $\widehat \nu_{i-1}<\left(\frac{M-2}2\right)^2$ and it is attained at a function $\psi_{i-1}\in \mathcal H_{0,M}$ we can define 
\begin{equation}\label{radial-singular-M}
  \widehat{\nu}_{i}:=\inf_{\substack{\psi\in\mathcal{H}_{0,M}\ \psi\neq 0\\ \psi\underline \perp_{M}\{\psi_1,\dots,\psi_{i-1}\}}}
  \frac{ \mathcal Q_{w}(\psi)  }{\int_0^1 t^{M-3}w^2\, dt}
\end{equation}
where the orthogonality stands for the orthogonality in $\mathcal L_M$. Again $  \widehat{\nu}_{i}$ is attained as far as $\widehat \nu_i <\left(\frac{M-2}2\right)^2$.\
The definitions, the properties of the eigenfunctions $\psi_i$ their behavior at $t=0$ and many other facts that we need in the following have been tackled in \cite{AG-part1}. Here we report only some properties of particular interest. The first one is called {\em Property 5} in \cite{AG-part1} and we recall it in a form that can be adapted both to the singular and the {\em classical} eigenvalues. 

\

\noindent {\bf{Property 5.}} {\em{Each singular eigenvalue $\widehat\nu_i$ (each eigenvalue $\nu_i$)
is simple and any $i$-th eigenfunction has exactly $i$ nodal domains.}}

\begin{proposition}[Proposition 3.11 in \cite{AG-part1}]\label{prop-prel-2}
	The number of negative eigenvalues $\nu_i$ defined in \eqref{Rayleigh-rad-M}  coincides with the number of negative eigenvalues $ \widehat \nu_{i}$  defined in \eqref{radial-singular-M}. 
\end{proposition}

\
	Eventually we go back to problem \eqref{general-f-H}: if $u$ is a radial solution and $w$ is defined as in \eqref{transformation-henon}, we can compute the Morse index of $u$ in terms of the singular eigenvalues $\widehat \nu_i$ of \eqref{radial-singular-problem-M} with $M$ given by \eqref{Malpha}.
	
\begin{proposition}[Proposition 1.4 in \cite{AG-part1}]\label{general-morse-formula-H}
Assume that $\a\geq 0$ and $f$ satisfies $H.1$ and take $u$ a radial weak solution to \eqref{general-f-H} satisfying $H.2$. Then its radial Morse index $m_{\rad}$ is the number of negative eigenvalues of \eqref{radial-singular-problem-M}, and its 
Morse index is given by
\begin{align}\label{tag-2-H}\begin{split}
m(u) = & \sum\limits_{i=1}^{m_{\rad}}\sum\limits_{j=0}^{\lceil J_i -1\rceil } N_j, \quad \qquad \mbox{where} \\
J_i= & \frac{2+\a}{2} \left(\sqrt{\left(\frac{M-2}{2}\right)^2- \widehat{\nu}_i}-\frac{M-2}{2}\right).
\end{split}\end{align}
Furthermore the negative singular eigenvalues are $\widehat\L = \left(\frac{2+\a}{2}\right)^2\widehat\nu_i +\lambda_j$ 
and the related eigenfunctions are, in spherical coordinates,
\begin{equation}\label{decomp-autofunz-f(u)-H}
\psi(x)= \phi_i\big( r^{\frac {2+\a}2}\big) Y_j(\theta), 
\end{equation} 
where $ \phi_i$ is an eigenfunction for \eqref{radial-singular-M}related to $\widehat\nu_i$.	
\end{proposition}

To characterize degeneracy, and in particular radial degeneracy, also the {\it classical} eigenvalues $\nu_i$ of \eqref{radial-singular-problem-M}, again with $M$ given by \eqref{Malpha}, are needed.

\begin{proposition}[Proposition 1.5 in \cite{AG-part1}]\label{non-degeneracy-H}
Assume that $\a\geq 0$ and $f$ satisfies $H.1$ and take $u$ a radial weak solution to \eqref{general-f-H} satisfying $H.2$.	When $N\ge 3$ then $u$ is radially degenerate if and only if $\widehat{\nu}_k = {\nu}_k =  0$  for some $k\ge 1$, 
and degenerate if and only if, in addition,
\begin{equation}\label{non-radial-degeneracy-H}
\widehat{\nu}_k =  - \left(\frac{2}{2+\a}\right)^2 j (N-2+j) \qquad \mbox{ for some $k, j\ge 1$.}
\end{equation}
Otherwise if $N=2$ then $u$ is radially degenerate if and only if $\nu_k=0$ for some $k\ge 1$, 
and degenerate if and only if, in addition, \eqref{non-radial-degeneracy-H} holds. \\
Besides in any dimension $N\ge 2$,  any nonradial function in the kernel of $L_u$ has the form \eqref{decomp-autofunz-f(u)-H}.
\end{proposition}

\section{Morse index of radial solutions }
\label{se:4}
In this section we address to the Morse index of radial solutions to the semilinear problem \eqref{general-f-H} 
when $\Omega$ is the unit ball, namely
\begin{equation}\label{general-f-H-5}
\left\{\begin{array}{ll}
-\Delta u = |x|^{\alpha}f(u) \qquad & \text{ in } B, \\
u= 0 & \text{ on } \partial B,
\end{array} \right.\end{equation}
where $\a\ge 0$ is a real parameter and $f$ satisfies  H.1. The case of $\a=0$  gives back the autonomous problem \eqref{general-f-auto} in $B$ and will be treat together with the general case. \\ 
As recalled in Section \ref{sec:2} any radial solution $u$ to \eqref{general-f-H-5} is linked by the transformation \eqref{transformation-henon} to a solution $w$ to \eqref{ode-sect-2} and \eqref{bc} with $M\geq 2$ given by \eqref{Malpha}.

To prove Theorem \ref{morse-estimate-H} we need some qualitative properties of solutions to semilinear O.D.E \eqref{ode-sect-2}.
 Let us denote by $0<t_1<\dots<t_m=1$ the zeros of $w$ in $[0,1]$, so that $w(t_i)=0$ and,   assuming $w(0)>0$ we let
\begin{align*}
{\mathcal M}_0 & = \sup \{ w(t) \ : \ 0<t<t_1 \}, \\
{\mathcal M}_i & = \max \{ |w(t)|\, : \, t_{i} \le r \le t_{i+1} \} ,
\end{align*}
for $i=1,\dots,m-1$. Then we have:
\begin{lemma}\label{lemma-max-assoluti}
 Assume that $\a\geq 0$ and  $f$ satisfies H.1 and let $w$ be a weak solution to \eqref{ode-sect-2} with $m$ nodal zones which is positive in the first one (starting from $0$) satisfying H.2. If in addition $f$ satisfies    $f(s)/s>0$ as $s\neq 0$, 
   then	 $w$ is strictly decreasing in its first nodal zone so that 
	\[ w(0)= {\mathcal M}_0 .\] 
Moreover it has a unique critical point $s_i$ in the nodal set $(t_{i}, t_{i+1})$ for $i=1,\dots m-1$  with 
	\[{\mathcal M}_0 > {\mathcal M}_2 > \dots \]
        \[{\mathcal M}_1 > {\mathcal M}_3 > \dots .\]
	In particular $0$ is the global maximum point and $s_1$ is the global minimum point. \\
	If, in addition $g$ is odd, then
        \[{\mathcal M}_0 > {\mathcal M}_1 > \dots {\mathcal M}_{m-1}.\]
\end{lemma}
\begin{proof}
Under assumptions H.1 and H.2 a weak solution to \eqref{ode-sect-2} and \eqref{bc} is classical by \cite[Corollary 4.8]{AG-part1}. Then integrating \eqref{ode-sect-2} and recalling that  $w>0$ in $(0,t_1)$ gives
	\[
		w'(t) = -\left(\frac{2}{2+\a}\right)^2 {t}^{1-M}
		\int_0^{t}s^{M-1}\frac{f(w)}{w} w  \ ds<0 
              \]
              for any $t\in (0,t_1)$.
	 Then $w$ is strictly decreasing in the first nodal zone, so that ${\mathcal M}_0 =w(0)$. 
	We multiply $-w''-\frac{M-1}t w'=\left(\frac{2}{2+\a}\right)^2 f(w)$ by $w'$ and integrate to compute 
        \begin{equation}\label{passaggio-2}
        \frac 12 \left(w'(t) \right)^2+(M-1)\int_0^t \frac{\left(w'(s) \right)^2}s\, ds=\left(\frac{2}{2+\a}\right)^2\left(F(w(0))-F(w(t))\right)\end{equation}
      where $F(s)=\int^s f(t) dt$ is a primitive of $f$. Since the l.h.s. is strictly positive, it follows that $F(w(0))>F(w(r))$ for any $t\in (0,1]$, meaning that $w(0)\neq w(t)$ for any $t\in (0,1]$. This implies that ${\mathcal M}_0=w(0)>w(t)$ for any $t\in (0,1]$ so that $0$ is the global maximum point of $w$. The very same computation (integrating between $t_i, t$) shows that $|w|$ is strictly increasing in any nodal region until it reaches a critical point $s_i$, and then it is strictly decreasing. At any critical point $s_i$, we have $w(s_i)\neq 0$ by the unique continuation principle and 
          $w''(s_i) = -\left(\frac{2}{2+\a}\right)^2 f(w(s_i))\neq 0$  has the same sign of $w(s_i)$ because $f(s)/s>0$, so that $w$ can have only one strict maximum point (resp. minimum) in each nodal set where it is positive (resp. negative). Further the previous argument also shows that ${\mathcal M}_0 > {\mathcal M}_{2} > \dots$ and that ${\mathcal M}_1 > {\mathcal M}_{3} > \dots$. If, in addition $g$ is odd, then $G$ is even and \eqref{passaggio-2} shows that $F(w(0))>F(|w(t)|)$ for any $t\in (0,1]$ from which it follows that ${\mathcal M}_0 > {\mathcal M}_1 > \dots {\mathcal M}_{m-1}.$    
\end{proof}

Next we show an estimate on $u'$ and $w'$ that will be useful in the following.
\begin{lemma}\label{lem:w'H1}
 Assume  that $\a\geq 0$ and $f$ satisfies H.1, take $u$ a radial weak solution to \eqref{general-f-H-5} satisfying H.2
and $w$ as in \eqref{transformation-henon}. Then $u'\in \mathcal H_N$ and $w'\in \mathcal H_M$. 
\end{lemma}
\begin{proof}
We prove that $u'\in \mathcal H_N$. The fact that $w'\in \mathcal H_M$ then follows by Lemma 4.4 and (4.21) in \cite{AG-part1}. By \cite[Lemma 4.6]{AG-part1} it is known that any weak solution $u\in C^2[0,1]$ and solves \eqref{ode-sect-2} in classical sense. In particular  $u''\in C[0,1]$ so that $\int_0^1 r^{N-1} |u''|^2 dr <\infty$.
	Moreover for every $\gamma<1+\a$ de L'Hopital Theorem gives
\[\lim_{r\to 0}\frac{u'(r)}{r^\gamma}=\lim_{r\to 0}\frac{r^{N-1}u'(r)}{r^{N-1+\gamma}}=\frac{-r^{1+\a-\gamma}f(u(r))}{N-1+\gamma}=0
\]
which shows that $\int_0^1 r^{N-3}|u'|^2\ dr<\infty$ and concludes the proof.
\end{proof}

The transformation \eqref{transformation-henon} is useful also in computing the Morse index of radial solutions $u$ to \eqref{general-f-H-5} via Proposition \ref{general-morse-formula-H}. In that case we look at  the singular eigenvalues $\widehat \nu_i$  defined in \eqref{radial-singular-problem-M} in Section \ref{sec:2}.
Next Proposition establishes some bounds for these singular eigenvalues $\widehat \nu_i$  which are essential to prove Theorem \ref{morse-estimate-H}.
  \begin{proposition}\label{stima-nu-H}
  Assume   that $\a\geq 0$ and $f$ satisfies H.1 and take $u$ a radial weak solution to \eqref{general-f-H-5} with $m$ nodal zones satisfying H.2.
	Then
	\begin{align}
	\label{nl<k-general-H} & \widehat{\nu}_i  < -(M-1)  \quad \text{ as } i=1,\dots m-1 .
	\intertext{If, in addition, $f(s)/s >0$ when $s\neq 0$  and the radial Morse index of $u$ is $m_\rad(u)\ge m$  then}
	\label{num>k-general-H} & 0>\widehat{\nu}_i  > -(M-1)  \quad \text{ as } i=m,\dots m_{\rad}(u) .
	\end{align}
\end{proposition}
  
  \begin{proof}
Let $w$ be as in \eqref{transformation-henon} and  $\zeta=w'\in C^1[0,1]\cap {\mathcal H}_{M}$ by Lemma \ref{lem:w'H1}. Since $w\in C^2[0,1]$ and satisfies \eqref{ode-sect-2} pointwise, a trivial computation shows that
                \begin{equation}
	      \label{A}
	      \int_0^1 r^{M-1}\zeta'\varphi' \ dr=\left(\frac 2{2+\a}\right)^2 \int_0^1 r^{M-1}f'(w) \zeta\varphi \ dr- (M-1)\int_0^1 r^{M-3} \zeta\varphi \ dr\end{equation}
                for any $\varphi \in C^{1}_0(0,1)$. 
	Moreover the computations in \cite[Lemma 2.4]{AG-part1} can be repeated obtaining that
		\begin{equation}\label{C}
		\left(r^{M-1} \left(\psi_i'\zeta - \psi_i \zeta'\right)\right)' = - (M-1+\widehat \nu_i) r^{M-3} \psi_i\zeta \ \text{ for $r\in (0,1)$}
		\end{equation}
		whenever $\psi_i$ is an eigenfunction for \eqref{radial-singular-problem-M}  related to $\widehat{\nu}_i<\left(\frac{M-2}{2}\right)^2$. 
		\\
		It is clear that $\zeta$ has at least $m$ zeros in $[0,1]$, indeed since $u$ has $m$ nodal domains the same is true for $w$ so that $\zeta$ has at least one zero in each nodal domain of $w$. Let $0\le t_0<t_1\dots <t_{m-1}\le 1$ be such that $\zeta(t_i)=0$. Because $w$ is a nontrivial solution to  \eqref{ode-sect-2} and \eqref{bc} we can take $t_0=0$, and certainly $t_{m-1}<1$ by the unique continuation principle. 
                For $k=1,\dots m-1$, let $\zeta_k$ be the function that coincides with $\zeta$ on $[t_{k-1}, t_k]$ and is null elsewhere. Certainly $\zeta_k \in {\mathcal H}_{0,N}\subset H^1_{0,N}$, and can be used as test function in \eqref{A} giving 
                	\begin{equation}\label{B}
		\int_0^1 t^{M-1}\left( (\zeta_k')^2 - \left(\frac 2{2+\a}\right)^2
		f'(w) \zeta_k^2\right) dt = - (M-1)	\int_0^1 t^{M-3} \zeta_k^2 dt <0 .
		\end{equation} 
		Recalling that $\zeta_k$ have contiguous supports and so they are orthogonal in $L^2_M$ (see Section \ref{sec:2} for the definition of the space),
		 \eqref{B} implies in the first instance that the quadratic form  $\mathcal Q_{w}$ in \eqref{forma-quadratica-a-rad}  is negative defined in the $m-1$-dimensional space spanned by $\zeta_1,\dots,\zeta_{m-1}$ showing, by \eqref{Rayleigh-rad-M}, that the eigenvalue problem \eqref{radial-eigenvalue-problem-M}  has at least $m-1$ negative eigenvalues $\nu_1, \dots, \nu_{m-1}$.
		 Proposition \ref{prop-prel-2} then implies that also the singular eigenvalue problem \eqref{radial-singular-problem-M} 
		 has at least $m-1$ negative eigenvalues $\widehat \nu_1, \dots, \widehat \nu_{m-1}$.
	Let us check that actually $\widehat\nu_i<-(M-1)$.
                First $ \widehat\nu_i\neq -(M-1)$, otherwise \eqref{C} should imply that $\psi_i$ and $\zeta$ are proportional, which is not possible as $\psi_i(1)=0\neq \zeta(1)$. Next, taking advantage from the identity \eqref{C},
                we can repeat the same arguments used to prove the { last part of Property 5 in Subsection 3.1 in \cite{AG-part1}} to show that, if $\widehat{\nu}_i > -(M-1)$, then $\psi_i$ must have at least one zero
                  between any two consecutive zeros of $\xi$ meaning that $\psi_i$ must have at least $m-1$ internal zeros, contradicting  Property 5 recalled in Section \ref{sec:2}. This concludes the proof of \eqref{nl<k-general-H}.
		
		Further when $f(s)/s>0$ as $s\neq 0$, then $w$ has only one critical point in any nodal region by Lemma \ref{lemma-max-assoluti}. 
		This means that the function $\zeta$ has exactly $m$ zeros, and only $m-1$ internal zeros. 
	Besides, since we are taking that $m_{\rad}(u)\ge m$, also $\widehat\nu_m<0$ thanks to Proposition \ref{general-morse-formula-H} and the related eigenfunction $\psi_m$ has $m$ nodal zones by the Property 5 recalled in Section \ref{sec:2}.
		The inequality \eqref{num>k-general-H} is obtained by comparing $\zeta$ and $\psi_m$. 
		As before certainly $ \widehat\nu_m\neq -(M-1)$, and if $ \widehat\nu_m< -(M-1)$ then  $\zeta$ must have at least $m$ internal zeros, obtaining a contradiction.	
\end{proof}

{ The previous inequalities will play a role in the proof of some asymptotic results on the Morse index of radial solutions to \eqref{general-f-H-5} in \cite{AG18,AG18-2}.}
Now the statement of Theorem \ref{morse-estimate-H} follows by combining the estimate \eqref{nl<k-general-H} with the general formula \eqref{tag-2-H}.

\begin{proof}[Proof of Theorem \ref{morse-estimate-H}]
By \eqref{nl<k-general-H}, via Proposition \ref{general-morse-formula-H}, it is clear that the radial Morse index of $u$ is at least $m-1$, i.e.~\eqref{general-radial-morse-H} holds.
Next putting the estimate \eqref{nl<k-general-H} inside \eqref{tag-2-H} gives \eqref{general-morse-estimate-H}.
\\
Moreover under assumption H.3 it is easy to see that the radial Morse index of $u$ is at least equal to the number of nodal zones.
First we show that, letting $w$ as in \eqref{transformation-henon}, the eigenvalue problem \eqref{radial-eigenvalue-problem-M} 
has at least $m$ negative eigenvalues i.e., by the variational characterization  \eqref{Rayleigh-rad-M}, that the quadratic form $\mathcal Q_{w}$ in \eqref{forma-quadratica-a-rad}   is negative defined in an $m$-dimensional subspace of $H^1_{0,M}$.
Let $0<t_1<t_2<\dots t_m=1$ be the zeros of $w$ in $[0,1]$,  $I_1=(0,t_1)$, $I_i=(t_{i-1},t_i)$ for $i=2,\dots, m$ its nodal domains, and $z_i$ be the function that coincides with $w$ in $I_i$ and is zero elsewhere.
	Using $z_i$ as a test function in \eqref{lane-emden-radial-weak-sol}
	gives
	\[ \begin{split}
	\int_0^1t^{M-1}\!\!
	\left( |z_i'|^2\!
	-\! \left(\frac 2{2+\a}\right)^2 \!\!\!
	f'(w)z_i^2\right) \! dt
	=\left(\frac 2{2+\a}\right)^2 \!\!\!\int_{I_i} t^{M-1}  
\!	\left(\dfrac{f(w)}{w} - f'(w )\right) w^2 dt < 0 \end{split}
	\]
	by H.3. So this part of the proof is concluded, because $z_i\in H^1_{0,M}$  are linearly independent, having contiguous supports.
 Proposition \ref{prop-prel-2} then implies that also the singular eigenvalue problem 
		   \eqref{radial-singular-problem-M}
		 has at least $m$ negative eigenvalues  and Proposition \ref{general-morse-formula-H} yields  that the radial Morse index of $u$ is at least $m$, i.e.~\eqref{radial-morse-f(u)-H} holds.
		Eventually \eqref{morse-estimate-f(u)-H} follows inserting \eqref{radial-morse-f(u)-H} into \eqref{general-morse-estimate-H}.
\end{proof}

  Theorem \ref{morse-estimate-H} extends some previous results on the autonomous case, namely \eqref{general-f-H-5} for $\a=0$, to the case of positive values of $\a$. 
 The proof above is nevertheless a new proof also for the autonomous case,  based upon the singular eigenvalue problem associated with the linearized operator $L_u$. Indeed when $\a=0$ the eigenvalues $\widehat\nu_i$ coincide with the radial singular eigenvalues $\widehat\L^{\rad}_i$ defined in \eqref{radial-singular-i} and  \eqref{nl<k-general-H} and \eqref{num>k-general-H}  become
\begin{align}\label{nl<k-general} 
	& \widehat\L_i^{\rad}  < -(N-1)  & \text{ as } i=1,\dots m-1 \\
	\label{num>k-general} 
	& 0>\widehat\L_i^{\rad}  > -(N-1)  &\text{ as } i=m,\dots m_{\rad}(u)  
	\end{align}
Some comments on estimates \eqref{nl<k-general} and \eqref{num>k-general}, which are important in providing the bound \eqref{general-morse-estimate-H} on the Morse index of $u$ in the case of $\a=0$. Indeed they imply  that  the parameters $J_i$ appearing in \eqref{tag-2} satisfy $J_i >1$ for $i=1,\dots,m-1$ and $J_i<1$ for $i=m,\dots,m_\rad(u)$. It means that the eigenvalues $\widehat \L_i^\rad$ for $i=m,\dots,m_\rad(u)$ give only the radial contribution (corresponding to $j=0$) to the Morse index of $u$, while the eigenvalues $\widehat \L_i^\rad$ for $i=1,\dots,m-1$ give always also the contribution corresponding to $j=1$.
  \\
 In the general case $\a>0$  the estimate \eqref{nl<k-general-H} implies that $J_i >\frac{2+\a}{2}$  for $i=1,\dots,m-1$, highlighting the role of $\a$ and proving that the Morse index of any nodal radial solution  goes to $+\infty$ as $\a\to \infty$.

Furthermore estimate \eqref{nl<k-general-H}, together with \cite[Corollary 4.13]{AG-part1},  gives informations also on the Morse index of any radial solution in symmetric spaces.
If $\mathcal G$ is any subgroup of the  orthogonal  group $O(N)$ 
 we say that a function $\psi(x)$ is $\mathcal{G}$-invariant if 
\[
\psi(g(x))=\psi(x) \quad \forall \ x\in \Omega \quad \forall \ g \in \mathcal{G}.
\]
{ We denote by $H^1_{0,\mathcal G}$ the subset of $H^1_0(B)$ made up by $\mathcal G$-symmetric functions and by $m^{\mathcal G}(u)$ the Morse index of a solution $u$ when computed in the space $H^1_{0,\mathcal G}$.}
\begin{corollary}\label{cor:morse-estimate-G}
Take $\a\ge 0$ and $f$ satisfying H.1, and let $u$ be a radial solution to \eqref{general-f-H} with $m$ nodal zones such that H.2 holds. Then 
	\begin{align*}
	m^{\mathcal G}(u) & \geq (m-1)+(m-1)\sum_{j=1}^{[\frac{2+\a}2]}N_{j}^{\mathcal G} .
	\intertext{If also assumption H.3 holds true, then }
	m^{\mathcal G}(u) & \geq m+(m-1)\sum_{j=1}^{[\frac{2+\a}2]}N_{j}^{\mathcal G}.
	\end{align*}
\end{corollary}
Here $N_j^{\mathcal G}$ stands for the multiplicity of $j^{th}$ eigenvalue of the Laplace-Beltrami operator in $H^1_{0,\mathcal G}$.
\remove{\\
Besides under the additional assumption H.4 the minimization technique on the nodal Nehari set can be performed also in $H^1_{0,\mathcal G}$, ending with a nodal solution $u$ which belongs to $H^1_{0,\mathcal G}$ and has $m^{\mathcal G}(u)=2$. 
In that way Corollary \ref{cor:morse-estimate-G} ensures that the minimal energy nodal {\it and $\mathcal G$-symmetric} solution is not radial whenever $N_1^{\mathcal G}\neq 0$, for every $\a\ge 0$. As $\a$ increases, the condition under which the minimal energy nodal {\it and $\mathcal G$-symmetric} solution can be radial become more stringent, and it is expected that the multiplicity of nonradial solutions  increases.
This considerations are exploited in \cite{GI}, dealing with the Lane Emden problem in the disk, and in \cite{AG18-2}, \cite{A}, dealing with and the H\'enon problem.

\edz{\F La parte da Beside ..... la sposterei anch'essa in sezione 4}
\  }

\section{Power type nonlinearity: the standard H\'enon equation}\label{sec:4}
We focus here on the particular case $f(u)= |u|^{p-1}u$ where $p>1$ is a real parameter.
	For $\alpha>0$ we have the H\'enon problem
\begin{equation}\label{H}
\left\{\begin{array}{ll}
-\Delta u = |x|^{\alpha}|u|^{p-1} u \qquad & \text{ in } B, \\
u= 0 & \text{ on } \partial B,
\end{array} \right.
\end{equation}
but all the following discussion applies also to the case $\alpha=0$, i.e. to the Lane-Emden problem 
\begin{equation}\label{LE}
\left\{\begin{array}{ll}
-\Delta u = |u|^{p-1} u \qquad & \text{ in } B, \\
u= 0 & \text{ on } \partial B.
\end{array} \right.
\end{equation}

To begin with we see that problem \eqref{H} admits classical solutions with any given number of nodal zones under assumption H.2', namely when 
		the exponent $p$ satisfies 
\begin{equation}\label{cond-p-Henon} \begin{array}{lcl}
p\in (1,+\infty) &  & \text{ when $N=2$} , \\  p\in \left(1,p_{\a,N}\right) &  \ \text{with } p_{\a,N}=\frac{N+2+2\a}{N-2} & \text{ when  $N>2$. } \end{array}
\end{equation}
More precisely we show the following
\begin{proposition}\label{prop:henon}
Assume that $\a\geq 0$ and $p$ satisfies \eqref{cond-p-Henon}. Any weak radial solution to \eqref{H} is classical. For any $m>1$ problem \eqref{H} admits a unique radial solution $u$ which is positive in the origin and has $m$ nodal regions. Further $u$ is strictly decreasing in its first nodal zone and it has a unique critical point $\sigma_i$ in any nodal zone $(r_{i-1},r_{i})$. Moreover
		\[u(0)>|u(\sigma_1)|>\dots |u(\sigma_{m-1})|\]
		and $0$ is the global maximum point.
\end{proposition}
As in the previous section the proof relies on the transformation \eqref{transformation-henon} that we adapt here to the case of the power nonlinearity so to adsorb the constant. Then a minor variation on the previous discussion shows that 
\begin{corollary}\label{cor:1-henon}
Assume that $\a\geq 0$. $u$ is 
a (weak or classical) radial solution to \eqref{H} if and only if 
\begin{equation}\label{transformation-henon-no-c}
v(t)= \left(\frac{2}{2+\a}\right)^{\frac{2}{p-1}} u(r) , \qquad t= r^{\frac{2+\a}{2}}
\end{equation}
solves (in weak or classical sense)
\begin{equation}\label{LE-radial}
\begin{cases}
- \left(t^{M-1} v^{\prime}\right)^{\prime}= t^{M-1} |v|^{p-1} v  , \qquad  & 0< t< 1, \\
v'(0)=0, \;  v(1) =0 , & 
\end{cases}\end{equation}
where $M  =\frac{2(N+\alpha)}{2+\alpha}\in[2,N]$ as in \eqref{Malpha}.
\end{corollary}
Next we show that under assumption \eqref{cond-p-Henon} a bootstrap argument applies to radial solutions to \eqref{H} showing the regularity statement of Proposition \ref{prop:henon}. To simplify the notations we prove it for a weak solution $v$ to \eqref{LE-radial} provided that 
\begin{equation}\label{cond-p-Lane-Emden} \begin{array}{lcl}
p\in (1,+\infty) &  & \text{ when $M=2$} , \\  p\in \left(1,p_M\right) &  \ \text{with } p_M=\frac{M+2}{M-2} & \text{ when  $M>2$ .} \end{array}
\end{equation}
Assumption \eqref{cond-p-Lane-Emden} is the restatement of \eqref{cond-p-Henon} in terms of $M$ and highlights the fact that $p_M=p_{\a,N}$ is exactly the critical exponent for existence results. The regularity of $u$ then follows from the regularity of $v$ by Corollary \ref{cor:1-henon}.
\begin{lemma}\label{lem:regularity-henon}
Let $v$ be any  weak solution to \eqref{LE-radial}. Then $v\in C^2[0,1]$ and is a classical solution.
\end{lemma}
 \begin{proof}
	Here we prove that $v\in C[0,1]$. The finer regularity then follows by \cite[Corollary 4.8]{AG-part1}.
	Since $v\in H^1_{0,M}$ it is continuous on $(0,1]$ and differentiable a.e.  with	
		\begin{equation}
		\label{e1}	v(t) = - \int_t^1 v'(\tau) d\tau. 
		\end{equation}	
		So we only need to show that $v$ is bounded near at $t=0$.
		
		By the embedding of $H^1_{0,M}$ into $L^{p}_M$ (see \cite[Lemma 5.4]{AG-part1}) we have that $t^{M-1} |v|^p \in L^1(0,1)$, so starting from the weak formulation of \eqref{LE-radial} and using  the same arguments used to obtain the equation (2.12) in  \cite[Proposition 2.2]{AG-part1} we end up with
		\begin{equation}\label{e2} 
		v'(\tau) = - \tau^{1-M} \int_0^\tau s^{M-1} |v(s)|^{p-1} v(s)\, ds .
		 \end{equation}
		If $M=2$, the Radial Lemma in $H^1_{0,M}$ proved in \cite[Lemma 5.2]{AG-part1} states that $|v(s)|\le C |\log s|^{\frac{1}{2}}$, which inserted into \eqref{e2} gives
	\[ 	|v' (\tau)| \le   \tau^{-1} \int_0^\tau s |\log s|^{\frac{p}{2}}  ds \to 0 \quad \text{ as } \tau\to 0 ,\]
		proving  that $v$ is continuous. \\
		Otherwise if  $M>2$ putting together \eqref{e1} and \eqref{e2} gives
		\begin{equation}\label{e3}
		|v(t)| \le  \int_t^1  \tau^{1-M} \int_0^\tau s^{M-1} |v(s)|^{p}  ds \, d\tau .
		\end{equation}
		Next the same Radial Lemma states that $|v(s)|\le C s^{-\frac{M-2}{2}}$, which inserted into \eqref{e3} gives
	\[
	|v(t)| \le  C \int_t^1  \tau^{1-M} \int_0^\tau s^{M-1- p\frac{M-2}{2}} ds \, d\tau \le  C \int_t^1  \tau^{1- p\frac{M-2}{2}}  d\tau ,
	\]
	where $C$ stands for a constant that can change from line to line.
	If $p< 4/(M-2)$, we have obtained that $v(t)$ is bounded near at $t=0$ as wanted. If $p=4/(M-2)$, then $|v(t)|\le C \left(1 +|\log t|\right)$ and we can conclude as in the case $M=2$.
	If, else, $p>4/(M-2)$, we have
	\begin{equation}\label{e4}
	|v(t)| \le C (1 + t^{2- p\frac{M-2}{2}} )
	\end{equation} 
	with $2- p\frac{M-2}{2} > - \frac{M-2}{2}$, so we can start a bootstrap argument. Inserting \eqref{e4} into \eqref{e3} yields
	\[ 	|v(t)| \le  C \int_t^1  \tau^{1-M} \int_0^\tau s^{M-1} \left(1 + s^{2- p\frac{M-2}{2}}\right)^p ds \, d\tau \le C \int_t^1  \left(1 + \tau^{1+p\left(2- p\frac{M-2}{2}\right)}\right) 
\]
	and iteratively
		\[ 	|v(t)| \le C \int_t^1  \left(1 + \tau^{1+\beta_n}\right) d\tau   \quad \text{ for } \  \beta_n = 2 \sum\limits_{k=0}^n p^k - \frac{M-2}{2} p^{n+1}	.\]
		If at some step $\beta_n = -2$ we infer $|v(t)| \le C \left(1 +|\log t|\right)$ and conclude as in the case $M=2$.
		Otherwise it is certain that after a finite number of steps $\beta_n > 0$, implying that $v(t)$ is bounded near at $t=0$.
		Actually $\beta_n = 2 p^{n+1} \left( \sum\limits_{k=0}^n p^{-1-k} - \frac{M-2}{4} \right)$ and $\sum\limits_{k=0}^n p^{-1-k} - \frac{M-2}{4} \to \frac{1}{p-1} - \frac{M-2}{4} > 0$ because of \eqref{cond-p-Lane-Emden}.
		\end{proof}
	
\begin{remark}\label{remark:H'}
		The same arguments in the proof of Lemma \ref{lem:regularity-henon} show that H.2 holds for any weak radial solution to \eqref{general-f-H}, when the nonlinearity $f$ satisfies the hypothesis H.1' mentioned in the introduction.
	\end{remark}

Next we recall how a solution to \eqref{LE-radial} with $m$ nodal zones can be produced provided that
\eqref{cond-p-Lane-Emden} holds.  This proves the existence part in Proposition \ref{prop:henon} again by Corollary \ref{cor:1-henon}.
Let
  \[
    {\mathcal E}(v) = \frac{1}{2}\int_0^1 r^{M-1} |v'|^2 dr-\frac{1}{p+1}\int_0^1 r^{M-1} |v|^{p+1} dr,\]
  be the energy functional associated to \eqref{LE-radial} which is defined on $H^1_{0,M}$ for the embedding of $H^1_{0,M}$ into $L^{2^*_M}_{M}$ as $p$ satisfies \eqref{cond-p-Lane-Emden}, see Lemma 5.3 in \cite{AG-part1}, where by $L^q_M$ we denote the extension to $q>1$ of the Lebesgue space $L^2_M$ in Section \ref{sec:2} and $2^*_M=\frac{2M}{M+2}$. Then, critical points of ${\mathcal E}$ are solutions to \eqref{LE-radial} and lie on the Nehari manifold
\[	{\mathcal N} =\left\{ v\in H^1_{0,M} \, : \, \int_0^1 r^{M-1} |v'|^2 dr=\int_0^1 r^{M-1} |v|^{p+1} dr \right\}\]
The compactness of the previous  embedding implies also that the minimum of ${\mathcal E}$ on ${\mathcal N}$ 
is attained and produces for every $p$ a couple of solutions $v^-<0<v^+$ to \eqref{LE-radial} such that $v^+=-v^-$, so that \eqref{LE-radial} admits a 
unique (by \cite{NN}) positive solution. 
By such minimality property one can also deduce that its radial Morse index is at most one and since H.3  is satisfied, then it is exactly one, by \eqref{radial-morse-f(u)-H}.

Moreover, since the nonlinear term $f(u)=|u|^{p-1}u$ is odd then
problem \eqref{LE-radial} admits infinitely many nodal solutions.
{In particular for every positive integer $m$, one can produce a solution $v$  to \eqref{LE-radial} with 
	\begin{equation}\label{segno-in-zero}
	v(0)>0
	\end{equation}
	which has exactly $m$ nodal zones, namely such that	there are $0 <t_1< t_2< \dots  t_m=1$ with 	\[\begin{array}{cll}
	v(r) > 0 & \; \text{ as } \; 0<r<t_1, \; & v(t_{i})=0 ,\\
	(-1)^i v(r) > 0 & \; \text{ as } \, t_i<r<t_{i+1}, & \end{array} \]
	as $i=1,\dots m-1$.
	It can be done by the so called Nehari method (see, for instance,  \cite{BW93}), i.e.~by introducing the spaces 
	\begin{align*}
	H^1_{0,M}(s,t)& =\{v\in H^1_M \, : \, v(s)=0=v(t) \}, 
	\intertext{the energy functionals}
	{\mathcal E}_{s,t}(v)& = \frac{1}{2}\int_s^t r^{M-1} |v'|^2 dr-\frac{1}{p+1}\int_s^t r^{M-1} |v|^{p+1} dr,
	\intertext{and the Nehari sets}
	{\mathcal N}_{s,t}& =\left\{ v\in H^1_{0,M}(s,t) \, : \, \int_s^t r^{M-1} |v'|^2 dr=\int_s^t r^{M-1} |v|^{p+1} dr \right\} ,
	\end{align*}
	and solving the minimization problem
	\begin{equation}\label{neharimin}
	\Lambda(t_1,\cdots t_{m-1}) := \min\left\{  \sum\limits_{i=1}^m \inf\limits_{{\mathcal N}(t_{i-1},t_{i})} {\mathcal E} \, : \, 0=t_0<t_1<\cdots<t_m=1\right\}.
	\end{equation}
	Afterwards it can be checked like in \cite[Lemma 5.1]{BW93} that choosing $t_0,t_1,\dots t_m$ which realize \eqref{neharimin} and gluing together, alternatively, the positive and negative solution in the sub-interval $(t_{i-1}, t_i)$, gives a nodal solution to \eqref{LE-radial}.
		Requiring \eqref{segno-in-zero} is sufficient to identify $v$ by the uniqueness results in \cite{NN}. 
		
To conclude the proof of Proposition \ref{prop:henon} it is needed to prove the qualitative properties of the solution to \eqref{H}. Via Corollary \ref{cor:1-henon}, it suffices to check the analogous properties of the solution to \eqref{LE-radial}.
To state them we need some more notations and write
\begin{align*}
{\mathcal M}_0 & = \sup \{ v(t) \ : \ 0<t<t_1 \}, \\
{\mathcal M}_i & = \max \{ |v(r)|\, : \, t_{i} \le r \le t_{i+1} \} ,
\end{align*}
\[
0=t_0 <t_1< s_1 < t_2< \dots t_{m-1}<s_{m-1}< t_m=1,  \]
where $t_i$ are the zeros of $v$, any $s_i$ is the extremal point of $v$ restricted to the nodal region $(t_i, t_{i+1})$, and ${\mathcal M}_i$ the respective extremal value.

\begin{lemma}\label{prop-4.2-H}
	Let $v$ be a weak solution to \eqref{LE-radial} with $m$ nodal zones which is positive in the first one (starting from $0$).  Then	
	\[ v(0)= {\mathcal M}_0 , \qquad v'(0) = 0.\] 
	Besides $v$ is strictly decreasing in its first nodal zone and $s_i$ is the only critical point in the nodal set $(t_{i,p}, t_{i+1,p})$ for $i=1,\dots m-1$ with 
	\[{\mathcal M}_0 > {\mathcal M}_1 > \dots {\mathcal M}_{m-1}.\]
	In particular $0$ is the global maximum point.
\end{lemma}
It follows by Lemma \ref{lemma-max-assoluti} using Lemma \ref{lem:regularity-henon}. 

The Morse index and the degeneracy of a solution $u$ to \eqref{H} can be regarded considering the eigenvalues and singular eigenvalues $\nu_i$ and $\widehat \nu_i$ as in \eqref{radial-eigenvalue-problem-M} and \eqref{radial-singular-problem-M} which in terms of $v$ are given by 
\begin{equation}\label{radial-general-H-no-c-0}
	\left\{\begin{array}{ll}
	- \left(t^{M-1} \psi'\right)'- t^{M-1} p|v|^{p-1} \psi = t^{M-1} {\nu}  \psi & \text{ for } t\in(0,1)\\
	\psi'(0)=0, \ \ \psi(1)=0
	\end{array} \right.
	\end{equation}
and 
\begin{equation}\label{radial-general-H-no-c}
\left\{\begin{array}{ll}
- \left(t^{M-1} \phi'\right)'- t^{M-1} p|v|^{p-1} \phi = t^{M-3} \widehat{\nu}  \phi & \text{ for } t\in(0,1)\\
\phi\in  \mathcal H_{0,M} .
\end{array} \right.
\end{equation} 
Indeed  in the particular case of power nonlinearity we have $p\, |v|^{p-1}=\left(\frac{2}{2+\a}\right)^2 f'(w)$, recalling \eqref{transformation-henon} and \eqref{transformation-henon-no-c}.  
	\\
	Besides the radial solutions produced in Proposition \ref{prop:henon} satisfy in particular the assumption H.2, so that Propositions \ref{general-morse-formula-H} and \ref{non-degeneracy-H} apply. Eventually we end up with

\begin{corollary}
Assume that $\a\geq 0$ and $p$ satisfies \eqref{cond-p-Henon}. The radial singular eigenvalues for the linearized operator $L_u$ are 
\begin{equation}
\label{relazione-autov-no-c}
\widehat\L^{\rad}_{i} = \left(\frac{2+\alpha}{2} \right)^2 \widehat{\nu}_i <\left(\frac{N-2}{2} \right)^2 
\end{equation}
where $\widehat{\nu}_i <\left(\frac{M-2}{2} \right)^2 $ are the eigenvalues of
\eqref{radial-general-H-no-c},  and the Morse index formula \eqref{tag-2-H} holds corresponding to these
$\widehat\nu_i$.
$\psi_i\in\mathcal H_{0,N}$ is an eigenfunction related to $\widehat \L^{\rad}_i$ if and only if $\psi_i(r)=\phi_i(t)$, where $\phi_i\in\mathcal{H}_{0,M}$ is an eigenfunction for problem \eqref{radial-general-H-no-c} related to  $\widehat{\nu}_i$.
For any $N\geq 2$ $u$ is degenerate (but not radially degenerate) if and only if 
 \begin{equation}\label{rel-autov-deg}
        \widehat \nu_k=-\Big(\frac{2+\a}2\Big)^2j(N-2+j) \ \text{ for some }j,k\geq 1.
        \end{equation}
     $u$ is radially degenerate instead if and only if $\widehat{\nu}=0$ is an eigenvalue for \eqref{radial-general-H-no-c} when $N\geq 3$ or $\nu=0$ is an eigenvalue for \eqref{radial-general-H-no-c-0} when $N=2$.
All the corresponding eigenfunctions are as in \eqref{decomp-autofunz-f(u)-H}.
\end{corollary}

Before proving Theorem \ref{prop:ultimo-H}, we point out some useful properties of an auxiliary function.

\begin{lemma}\label{lem:zeri-zeta}
	Let $v$ be a weak solution to \eqref{LE-radial} with $m$ nodal zones and 
\begin{equation}
\label{funz-ausil-z}
z=r\, v '+\dfrac{2}{p-1}v.
\end{equation}
The function $z$ has exactly $m$ zeros in $(0,1)$.
\end{lemma}
\begin{proof}
 By  Lemma \ref{lem:w'H1} and   \cite[Corollary 4.8]{AG-part1} the function $z$ belongs to $H^1_{0,M} \cap C^1[0,1]$, and it is easily seen that  solves 
\begin{equation}\label{eq-z}
\left(r^{M-1}z'\right)'+ pr^{M-1}|v|^{p-1}z=0 
\end{equation}
{in the sense of distributions}. Next, as clearly $pr^{M-1}|v|^{p-1}z$ is at least continuous on $[0,1]$, the same reasoning of   \cite[Proposition 4.6]{AG-part1} proves that $z$ solves \eqref{eq-z} pointwise.
\\
Because of \eqref{segno-in-zero} $z(0)=v(0)>0$, $z(t_1)= t_1 v'(t_1) \le 0$ and similarly $(-1)^i z(t_i)=(-1)^i t_i v'(t_i) \ge 0$. Actually the unique continuation principle guarantees that $(-1)^i z(t_i)=(-1)^i t_i v'(t_i) > 0$, i.e.~$z$ has alternating sign at the zeros of $v$ and therefore it has an odd number of zeros in any nodal zone of $v$. The claim follows because $z$ can not have more than one zero in any nodal zone. 

To see this fact, it is needed to look back to the Nehari construction of the nodal solution $v$.  
By construction $w_0(x):= v(|x|)$ as $|x|\le t_1$ is the unique  positive radial solution to \eqref{H}  settled in the ball $\Omega=\{x\in \R^N\, : \,  |x| < t_{1}\}$ and therefore 
\begin{equation}\label{eigen-ball}
\left\{\begin{array}{ll}
- \left(t^{M-1} \phi'\right)'- t^{M-1} p|v|^{p-1} \phi = t^{M-1}{\nu}  \phi & \text{ for } t\in(0, t_{1})\\
\phi'(0)=\phi(t_{1})=0  
\end{array} \right.
\end{equation}
has exactly one negative eigenvalue $\nu_1$. 
\\
Similarly for $i=1,\dots m-1$ $w_i(x):=(-1)^{i} v(|x|)$ as $t_{i}\le r \le t_{i+1}$  is the unique  positive radial solution to \eqref{H} settled in the annulus $\Omega=\{x\in \R^N\, : \, t_{i} < |x| < t_{i+1}\}$ and then it realizes the minimum of ${\mathcal E}_{t_{i},t_{i+1}}$.
Again it follows that 
\begin{equation}\label{eigen-anello}
\left\{\begin{array}{ll}
- \left(t^{M-1} \phi'\right)'- t^{M-1} p|v|^{p-1} \phi = t^{M-1}{\nu}  \phi & \text{ for } t\in(t_{i}, t_{i+1})\\
\phi(t_i)=\phi(t_{i+1})=0  
\end{array} \right.
\end{equation}
has exactly one negative eigenvalue $\nu_1$.
\\
Now, let assume by contradiction that $z$ has three or more zeros between $t_i$ and $t_{i+1}$, and 
let $\phi_2$, ${\nu}_2$ respectively the second eigenfunction and eigenvalue of
\eqref{eigen-ball} or \eqref{eigen-anello} settled in $(t_i,t_{i+1})$. 
We have seen that 
${\nu}_2\geq 0$, and by the  analogous of Property 5, see Section \ref{sec:2} in the interval $(t_{i},t_{i+1})$ for $i\geq 0$ 
$\phi_2$ has exactly one zero in $(t_i,t_{i+1})$.
If $z$ has three or more zeros between $t_i$ and $t_{i+1}$, then we can reason exactly as in the proof of Property 5 of Subsection 3.1 of \cite{AG-part1} and we prove that $\phi_2$ has at least two zeros in the same interval obtaining a contradiction. 
To see this   we take that $z(r)>0$ on $(s_{1},s_2)$ with $z(s_1)=z(s_2)=0$, which also implies $z'(s_{1})>0$ and $z'(s_{2})<0$. If $\phi_{2}$ does not vanishes inside $(s_{1},s_2)$ we may assume without loss of generality that 
$\phi_{2}(r)>0$ in $(s_{1},s_2)$ and $\phi_{2}(s_{1}), \phi_{2}(s_{2})\ge 0$. 
Repeating the computations in Lemma 2.4 in \cite{AG-part1} we get that 
	\begin{equation}\label{CC}
	\left(r^{N-1} \left(z'\phi_2 - z \phi_2'\right)\right)' = \nu_2 r^{N-1} z \phi_2  \qquad \text{ as } t_i < r < t_{i+1}.
	\end{equation}
Integrating \eqref{CC} on $(s_{1},s_2)$  gives 
\[ s_2^{M-1} z'(s_2)\phi_{2}(s_2)- s_{1}^{M-1}  z'_i(s_{1})\phi_{2}(s_{1}) = \nu_{2}\int_{s_{1}}^{s_2} r^{M-1} z \phi_{2} \ dr. \]
But this is not possible because the l.h.s. is less or equal than zero  by the just made considerations, while the r.h.s. is greater or equal than zero  as $\nu_{2}\geq 0$. The only possibility is that $\nu_2=0$ and $\phi_2(s_1)=\phi_2(s_2)=0$, but again this is not possible since it implies, by uniqueness of an eigenfunction, that $\phi_2$ and $z$ are multiples and this does not agree with $\phi_2(t_i)=0\neq z(t_i)$. 
\end{proof}

We are now in the position to prove Theorem \ref{prop:ultimo-H}:
	$u$ has radial Morse index $m$ and it is radially non-degenerate

\begin{proof}[Proof of  Theorem \ref{prop:ultimo-H}]
	First \eqref{radial-morse-f(u)-H} assures that $m_{rad}(v)\ge m$  which implies, in turn, that $\nu_i < 0$ as $i=1, \dots m$ by  Propositions \ref{general-morse-formula-H} and \ref{prop-prel-2}.\\
       The proof is completed if we show that 
       $\nu_{m+1}>0$. Indeed in this case Proposition \ref{prop-prel-2} forbids $\widehat\nu_{m+1}<0$, thus implying that $m_{\rad}(u)=m$ via Proposition \ref{general-morse-formula-H}, while Proposition \ref{non-degeneracy-H} ensures that $u$ is not radially degenerate.
       	We therefore assume by contradiction that $\nu_{m+1}\leq 0$ and  denote by $\psi_{m+1}$ the corresponding eigenfunction, which, by Property 5 in Section \ref{sec:2}
	admits $m$ zeros inside the interval $(0,1)$ and then $m+1$ nodal zones. Then we want to prove that the function $z$ introduced in 
          \eqref{funz-ausil-z}
          has at least one zero in any nodal interval of $\psi_{m+1}$. This fact contradicts Lemma \ref{lem:zeri-zeta}, since $z$ has $m$ zeros in $(0,1)$ and concludes the proof. Let $(s_k,s_{k+1})$ be a nodal zone for $\psi_{m+1}$ and suppose by contradiction that $z$ has one sign in this interval. Without loss of generality we can assume $\psi_{m+1}>0$ in $(s_k,s_{k+1})$, which 
also implies $\psi_{m+1}'(s_{k})>0$ and $\psi_{m+1}'(s_{k+1})<0$. If $z$ does not vanishes inside $(s_{k},s_{k+1})$ we may assume without loss of generality that 
$z(r)>0$ in $(s_{k},s_{k+1})$ and $z(s_{k}), z(s_{k+1})\ge 0$.  The arguments in the proof of Lemma  2.4 in \cite{AG-part1}  yield
\begin{equation}\label{genoveffa} \left(r^{M-1}\left(\psi_{m+1}'z- \psi_{m+1}z'\right) \right)' = - \nu_{m+1} r^{M-1} \psi_{m+1} z ,
\end{equation}
	and integrating on $(s_{k},s_{k+1})$  gives 
\[ s_{k+1}^{M-1} \psi_{m+1}'(s_{k+1})z(s_{k+1})- s_{k}^{M-1}  \psi_{m+1}'(s_{k})z(s_{k}) = -\nu_{m+1}\int_{s_{k}}^{s_{k+1}} r^{M-1}  \psi_{m+1} z \ dr. \]
Observe that the the r.h.s. is strictly positive if $\nu_{m+1}<0$ and equal to zero if $\nu_{m+1}=0$, while the l.h.s. is less or equal than zero by the assumptions on $z$ and $\psi_{m+1}$. The only possibility is that $\nu_{m+1}=0$ and $z(s_k)=z(s_{k+1})=0$.   So \eqref{genoveffa} implies that $\psi_{m+1}$ and $z$ are multiples and it is not possible since $\psi_{m+1}(1)=0\neq z(1)$. \end{proof}

\begin{remark}\label{alpha=0}
		Inspecting all the arguments used in this subsection one can easily see that they apply also to the case $\a=0$, i.e. to the Lane-Emden problem. In that particular case the transformation \eqref{transformation-henon-no-c} is the identity, and the presented proof of Theorem \eqref{prop:ultimo-H} is an alternative proof of \cite[Proposition 2.9]{HRS}.
		\end{remark}

{ We end this section recalling that when we are in a variational setting, namely when $1<p<\frac{N+2}{N-2}$, solutions to \eqref{general-f-H-5} (radial and nonradial) can be found minimizing
 the functional
\[\mathcal E(u):=\int_B \left(|\nabla u|^2-|x|^\a|u|^{p+1} \right)dx\]
(which is defined in $H^1_0(B)$)  under suitable constraints. In particular minimizing it on the Nehari manifold produces a least energy solution which is positive and not radial 
 when $\a$ is sufficiently large (depending on $p$) by the result in \cite{SSW}. 
 	Next following \cite{Bartsch-Weth} one can minimize $\mathcal E(u)$ on the nodal Nehari manifold to produce a nodal least energy solution which has two nodal domains and Morse index 2, and considerations based on the Morse index imply that such solution is not radial for $\a=0$, see \cite{AP} and \cite{BDG}.
Estimate \ref{general-morse-estimate-H} then extends this matter also to the case $\a>0$, proving Corollary \ref{cor-ultimo}. \\
Besides, if $\mathcal G$ is any subgroup of $O(N)$, for $1<p<\frac{N+2}{N-2}$, the minimization technique on the nodal Nehari set can be performed also in $H^1_{0,\mathcal G}$, ending with a nodal solution $u$ which belongs to $H^1_{0,\mathcal G}$ and has $m^{\mathcal G}(u)=2$. 
In that way Corollary \ref{cor:morse-estimate-G} ensures that the minimal energy nodal {\it and $\mathcal G$-symmetric} solution is not radial whenever $N_1^{\mathcal G}\neq 0$, for every $\a\ge 0$. As $\a$ increases, the condition under which the minimal energy nodal {\it and $\mathcal G$-symmetric} solution can be radial become more stringent, and it is expected that the multiplicity of nonradial solutions  increases.
This considerations are exploited in \cite{GI}, dealing with the Lane Emden problem in the disk, and in \cite{AG18-2}, \cite{A}, dealing with and the H\'enon problem.

	}

\end{document}